\documentclass[12pt,a4paper,reqno]{amsart}
\usepackage[T1]{fontenc}
\usepackage{microtype, a4wide, textcomp,fbb}
\usepackage{
    amsmath,  amssymb,  amsthm,   amscd,
    gensymb,  graphicx, etoolbox, 
    booktabs, stackrel, mathtools    
}
\usepackage[colorlinks=true, linkcolor=blue, citecolor=blue, urlcolor=blue, breaklinks=true]{hyperref}
\usepackage[capitalise]{cleveref}
\usepackage{placeins}
\newtheorem{theorem}{Theorem}[section]

\newtheorem{lemma}[theorem]{Lemma}

\newtheorem{proposition}[theorem]{Proposition}
\theoremstyle{definition}
\newtheorem{definition}[theorem]{Definition}
\newtheorem{example}[theorem]{Example}

\newtheorem{remark}[theorem]{Remark}

\newtheorem{result}[theorem]{Result}

\usepackage{tikz}
\usetikzlibrary{arrows}
\usetikzlibrary{shapes}
\usepackage{float}
\usepackage{tabularx}
\usepackage{kantlipsum}
\usepackage{array}
\usepackage{ytableau}

\begin{document}

\title{Sum of squares of hook lengths and contents}

\author{Krishna Menon}
\address{Department of Mathematics, KTH Royal Institute of Technology, Stockholm, Sweden}
\email{puzhan@kth.se}


\begin{abstract}
It is known that for the Young diagram of any partition of an integer $n$, the sum of squares of the hook lengths of its cells is exactly $n^2$ more than that of the contents of its cells. 
That is, for any partition $\lambda$ of an integer $n$,
\begin{equation*}
    \sum_{u \in \lambda} h(u)^2 = n^2 + \sum_{u \in \lambda} c(u)^2.
\end{equation*}
We provide a bijective proof of this fact, thus solving a problem posed by Stanley. 
Along the way, we obtain a formula for the number of rectangles in the Young diagram of a partition. 
We also mention a result for sums of other powers of hook lengths and contents.
\end{abstract}

\maketitle

\section{Preliminaries}

A partition $\lambda$ of an integer $n$, denoted $\lambda \vdash n$, is a weakly decreasing sequence of non-negative integers $\lambda_1 \geq \lambda_2 \geq \lambda_3 \geq \cdots$ whose sum is $n$. 
We represent a partition $\lambda$ using its \emph{Young diagram}. 
This consists of rows of boxes with the $i^{th}$ row (from the top) consisting of $\lambda_i$ boxes. 
For example, the Young diagram of the partition $(5, 5, 4, 2)$ is given in \Cref{yd}.

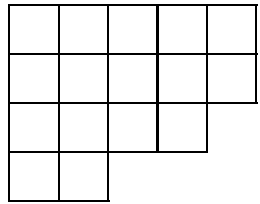
\begin{figure}[H]
    \centering
    \ydiagram{5, 5, 4, 2}
    \caption{Young diagram of $(5, 5, 4, 2)$.}
    \label{yd}
\end{figure}

We label the boxes, which we call \emph{cells}, just as one would label the entries of a matrix. 
For example, the cell $(3, 4)$ in the Young diagram in \Cref{yd} is the last cell in the third row.

For any cell $u = (i, j) \in \lambda$ (we often identify $\lambda$ with its Young diagram), the \emph{hook} of $u$, denoted $H(u)$, is the set of cells $(k, l) \in \lambda$ such that $k = i$ and $l \geq j$ or $l = j$ and $k \geq i$. 
The highlighted cells in the Young diagram in \Cref{hookfig} form the hook of the starred cell.

\begin{figure}[H]
    \centering
    \begin{ytableau}
    *(white) & *(white)  &*(white) & *(white) & *(white)\\
    *(white) & *(yellow) \star  &*(yellow) & *(yellow) & *(yellow)\\
    *(white) & *(yellow)  &*(white) & *(white) \\
    *(white) & *(yellow)
    \end{ytableau}
    \caption{Highlighted cells form the hook of the starred cell.}
    \label{hookfig}
\end{figure}

For any cell $u = (i, j)$, its \emph{hook length} is the size of $H(u)$, which we denote by $h(u)$. 
It can also be checked that $h(u) = \lambda_i + \lambda'_j - i - j + 1$ where $\lambda'$ is the \emph{conjugate} of $\lambda$. 
The partition $\lambda'$ is the one obtained by flipping the Young diagram of $\lambda$ along the diagonal (changing rows to columns and vice-versa). 
The \emph{content} of a cell $u = (i, j)$ is simply defined to be $c(u) = j - i$.

We have the following result connecting these two notions.

\begin{theorem}\label{thm}
    For any $n \geq 1$ and $\lambda \vdash n$, we have
    \begin{equation}\label{eq}
        \sum_{u \in \lambda} h(u)^2 = n^2 + \sum_{u \in \lambda} c(u)^2.
    \end{equation}
\end{theorem}

This identity appears as Problem 74(a) in the supplementary problems for \cite[Chapter 7]{ec2}. 
This statement also appears as Example 5 in \cite[Section I.1]{mac}. 
A computational solution by Zaimi \cite{312799} and an inductive solution by Hopkins \cite{312775} are available on MathOverflow.

We give a bijective proof of this result, thus providing a solution to Supplementary Problem 74(b) of \cite[Chapter 7]{ec2}.

\section{Bijective proof of \Cref{thm}}\label{proofsec}

Here and in the sequel, we fix $n \geq 1$ and  a partition $\lambda$ of $n$. 
To prove \Cref{thm}, we interpret each term in \Cref{eq} in the following manner (the example that follows might make these interpretations clearer):

\begin{enumerate}
    \item Note that $\sum_{u \in \lambda} h(u)^2$ is the cardinality of the set
    \begin{equation*}
        H(\lambda) := \{(u, v_1, v_2) \mid u \in \lambda, v_1, v_2 \in H(u)\}.
    \end{equation*}
    We represent each element of this set in the Young diagram of $\lambda$ by using a star to specify the cell $u$, $1$ for the cell $v_1$, and $2$ for the cell $v_2$.

    \item We interpret the term $n^2$ as counting ordered pairs of cells in the Young diagram of $\lambda$. 
    Again, we represent such a choice by placing $1$ in the first cell of the pair and $2$ in the second. 
    We use $N(\lambda)$ to denote the set of such pairs.

    \item To interpret the term $\sum_{u \in \lambda}c(u)^2$, we first break up the Young diagram of $\lambda$ using \emph{diagonal hooks}. 
    These are the hooks corresponding to the cells of $\lambda$ on the main diagonal, i.e., the hooks corresponding to cells of the form $(i, i)$ in $\lambda$. 
    Note that each cell is in a unique diagonal hook.
    
    Suppose that the cell $u$ is in the diagonal hook corresponding to $(i, i)$. 
    Note that if $u = (i, i)$, then $c(u) = 0$. 
    If $u$ is below $(i, i)$, then $|c(u)|$ is the number of cells in the diagonal hook that are in the same column as $u$ and strictly above it. 
    Similarly, if $u$ is to the right of $(i, i)$, then $|c(u)|$ is the number of cells in the diagonal hook that are in the same row as $u$ and strictly to the left of it.
    
    Hence, we interpret $\sum_{u \in \lambda}c(u)^2$ as the number of ways to choose a cell and then pick an ordered pair of cells to the left or above it in the diagonal hook in which it lies. 
    We represent such a choice by placing a star in the cell $u$ and, just as before, use labels $1, 2$ to represent the chosen ordered pair. 
    We denote the set of these choices by $C(\lambda)$. 
    That is, $C(\lambda)$ consists of the tuples $(u, v_1, v_2)$ where $u$ is a cell in $\lambda$ and $v_1, v_2$ are cells in the same diagonal hook as $u$ and that lie either
    \begin{itemize}
        \item in the same row and strictly to the left of $u$, or
        \item in the same column and strictly above $u$.
    \end{itemize}
\end{enumerate}

\begin{example}\label{objex}
Let $\lambda = (5, 5, 4, 2)$. 
The elements
\begin{itemize}
    \item $((1, 3), (1, 4), (3, 3)) \in H(\lambda)$, 
    \item $((2, 2), (2, 2)) \in N(\lambda)$, and
    \item $((2, 5), (2, 4), (2, 2)) \in C(\lambda)$
\end{itemize}
are shown in \Cref{objexfig}. 
The diagonal hooks have been colored in the last figure for clarity.

\begin{figure}[H]
    \centering
    \hfill
    \ytableaushort
    {\none \none \star 1 \none,\none, \none \none 2}
    * {5,5,4,2}
    \hfill
    \ytableaushort
    {\none, \none {\scriptstyle 1, 2} \none}
    * {5,5,4,2}
    \hfill
    \begin{ytableau}
    *(lime!60) & *(lime!60)  &*(lime!60) & *(lime!60) & *(lime!60)\\
    *(lime!60) & *(yellow) 2  &*(yellow) & *(yellow) 1 & *(yellow) \star\\
    *(lime!60) & *(yellow)  &*(pink) & *(pink) \\
    *(lime!60) & *(yellow)
    \end{ytableau}
    \hfill
    \caption{Elements of $H(\lambda), N(\lambda), C(\lambda)$ from \Cref{objex}.}
    \label{objexfig}
\end{figure}

\end{example}

We prove \Cref{thm} by exhibiting a bijection
\begin{equation*}
    H(\lambda) \rightarrow N(\lambda) \cup C(\lambda).
\end{equation*}
We actually describe three maps from three disjoint subsets of $H(\lambda)$ to $N(\lambda) \cup C(\lambda)$ which combine to give us the required bijection. 
In examples, we use red labels for elements in $N(\lambda) \cup C(\lambda)$ to distinguish them from elements of $H(\lambda)$.

We first describe a map between certain subsets $H_1(\lambda) \subseteq H(\lambda)$ and $N_1(\lambda) \subseteq N(\lambda)$. 
We define $H_1(\lambda)$ to be the set of those elements of $H(\lambda)$ where
\begin{itemize}
    \item the cells labeled $1$ and $2$ are neither in the same row nor in the same column, or
    \item at least one of the cells labeled $1$ or $2$ coincides with the cell containing the star.
\end{itemize}

The subset $N_1(\lambda)$ consists of those elements of $N(\lambda)$ where
\begin{itemize}
    \item the chosen cells are of the form $(k, l)$ and $(i, j)$ where $k > i$ and $l < j$, or
    \item the chosen cells are either in the same row or same column.
\end{itemize}

\begin{lemma}
    There is a bijection $\phi_1$ between $H_1(\lambda)$ and $N_1(\lambda)$.
\end{lemma}

\begin{proof}
    We describe how to construct an element of $N_1(\lambda)$ when given an element of $H_1(\lambda)$.
    \begin{enumerate}
        \item If the cells labeled $1$ and $2$ in an element of $H_1(\lambda)$ are neither in the same row nor in the same column, then we associate the element in $N_1(\lambda)$ obtained by simply deleting the star.
    
        \item If at least one of the cells labeled $1$ or $2$ coincides with the cell containing the star, then in this case as well, we associate the element in $N_1(\lambda)$ obtained by simply deleting the star.
    \end{enumerate}
    It can be checked that this gives us the required bijection.
\end{proof}

\begin{figure}[H]
    \centering
    \ytableaushort
    {\none \none \star 1,\none, \none \none 2}
    * {4,4,3,2}
    \hspace{0.3cm}
    $\xrightarrow{\phi_1}$
    \hspace{0.2cm}
    \ytableaushort
    {\none \none \none {\textcolor{red}{\textbf{1}}},\none, \none \none {\textcolor{red}{\textbf{2}}}}
    * {4,4,3,2}
    \hfill
    \ytableaushort
    {\none {\scriptstyle \star\ 1} \none, \none, \none 2}
    * {4,3,2}
    \hspace{0.3cm}
    $\xrightarrow{\phi_1}$
    \hspace{0.2cm}
    \ytableaushort
    {\none {\textcolor{red}{\textbf{1}}} \none, \none, \none {\textcolor{red}{\textbf{2}}}}
    * {4,3,2}
    \vspace{0.5cm}

    \hfill
    \ytableaushort
    {\none {\scriptscriptstyle \star 1, 2} \none, \none, \none}
    * {3,3,2}
    \hspace{0.5cm}
    $\xrightarrow{\phi_1}$
    \hspace{0.4cm}
    \ytableaushort
    {\none {\scriptstyle \textbf{\textcolor{red}{1, 2}}} \none, \none, \none}
    * {3,3,2}
    \hfill
    \caption{Elements of $N_1(\lambda)$ associated to certain elements of $H_1(\lambda)$ under $\phi_1$.}
    \label{hn}
\end{figure}

We now describe a map between another subset of $H(\lambda)$ and the remaining elements of $N(\lambda)$.

The elements of $H(\lambda) \setminus H_1(\lambda)$ are those where
\begin{itemize}
    \item all the labeled cells are either in the same row or same column, and
    \item the starred cell does not contain any other label.
\end{itemize}
We use $H_2(\lambda)$ to denote the set of those elements of $H(\lambda) \setminus H_1(\lambda)$ where the labels do \emph{not} lie in a single diagonal hook.

We define $N_2(\lambda)$ to be the set $N(\lambda) \setminus N_1(\lambda)$. 
Note that $N_2(\lambda)$ consists of those elements of $N(\lambda)$ where the chosen cells are of the form $(k, l)$ and $(i, j)$ where $k < i$ and $l < j$.

\begin{lemma}\label{phi2}
    There is a bijection $\phi_2$ between $H_2(\lambda)$ and $N_2(\lambda)$.
\end{lemma}

\begin{proof}
    Suppose we are given an element of $H_2(\lambda)$. 
    We will construct an element of $N_2(\lambda)$ such that the position of the south-east labeled cell coincides with that of the last labeled cell along the row or column of the given element of $H_2(\lambda)$.
    
    First suppose that the element of $H_2(\lambda)$ has the labels in distinct cells. 
    If the labels are in row $i$, in the cells $(i, k), (i, l), (i, j)$ where $k < l < j$, then move the label in cell $(i, l)$ to cell $(k, l)$ and delete the star (in cell $(i, k)$). 
    Note that $k < i$ is implied by the fact that all labels are not in the same diagonal hook. 
    Similarly, if the labels are in column $j$ in the cells $(k, j), (l, j), (i, j)$ where $k < l < i$, then move the label in cell $(l, j)$ to cell $(l, k)$ and delete the star.
    
    \Cref{hrn1} will hopefully make this map clearer. 
    Note that when the labels of the element of $H_2(\lambda)$ are in the same row, the label \emph{not} in the rightmost cell moves to a cell strictly above the diagonal. 
    When the labels are in the same column, the label not in the lowermost cell moves to a cell strictly below the diagonal.
    
    \begin{figure}[H]
        \centering
        \begin{tikzpicture}[rotate = 90, scale = 0.75]
            \draw [step = 1cm, black] (0, 0) grid (4, 6);
            \draw [dashed, ultra thin, gray] (4, 6) -- (0, 2);
            \node at (0.5, 0.5) {$\mathbf{\textcolor{red}{2}}\ 2$};
            \node (star) at (0.5, 4.5) {$\star$};
            \node (r1) at (2.5, 1.5) {$\mathbf{\textcolor{red}{1}}$};
            \node (1) at (0.5, 1.5) {$1$};
    
            \draw [thin, red] (star) -- (2.5, 4.5);
            \draw [->, thin, red] (2.5, 4.5) -- (r1);
            \draw [->, thin, red] (1) -- (r1);
    
            \node at (2, -0.75) {$\mathbf{\cdots}$};
            \node at (-0.75, 3) {$\mathbf{\vdots}$};
        \end{tikzpicture}
        \hspace{2cm}
        \begin{tikzpicture}[rotate = 90, scale = 0.75]
            \draw [step = 1cm, black] (0, 0) grid (5, 7);
            \draw [dashed, ultra thin, gray] (5, 7) -- (0, 2);
            \node at (0.5, 0.5) {$\mathbf{\textcolor{red}{1}}\ 1$};
            \node (2) at (1.5, 0.5) {$2$};
            \node (star) at (3.5, 0.5) {$\star$};
            \node (r2) at (1.5, 5.5) {$\mathbf{\textcolor{red}{2}}$};
    
            \draw [thin, red] (star) -- (3.5, 5.5);
            \draw [->, thin, red] (3.5, 5.5) -- (r2);
            \draw [->, thin, red] (2) -- (r2);
    
            \node at (2.5, -0.75) {$\mathbf{\cdots}$};
            \node at (-0.75, 3.5) {$\mathbf{\vdots}$};
        \end{tikzpicture}
        \caption{Methods for relabeling an element of $H_2(\lambda)$.}
        \label{hrn1}
    \end{figure}
    
    Next, suppose that the element of $H_2(\lambda)$ has labels $1, 2$ in the same cell $(i, j)$. 
    If the starred cell is in the same row, in cell $(i, k)$, then label the cell $(i, j)$ with $1$, the cell $(k, k)$ with $2$, and delete the star. 
    Similarly, if the starred cell is in the same column, in cell $(k, j)$, then label the cell $(i, j)$ with $2$, the cell $(k, k)$ with $1$, and delete the star.
    
    \begin{figure}[H]
        \centering
        \begin{tikzpicture}[rotate = 90]
            \draw [step = 1cm, black] (0, 0) grid (4, 6);
            \draw [dashed, ultra thin, gray] (4, 6) -- (0, 2);
            \node at (0.5, 0.5) {$\mathbf{\textcolor{red}{1}}\ {\scriptstyle 1, 2}$};
            \node (star) at (0.5, 4.5) {$\star$};
            \node (r1) at (2.5, 4.5) {$\mathbf{\textcolor{red}{2}}$};
    
            \draw [->, thin, red] (star) -- (r1);
    
            \node at (2, -0.75) {$\mathbf{\cdots}$};
            \node at (-0.75, 3) {$\mathbf{\vdots}$};
        \end{tikzpicture}
        \hspace{2cm}
        \begin{tikzpicture}
            \draw [step = 1cm, black] (0, 1) grid (4, 6);
            \draw [dashed, ultra thin, gray] (0, 6) -- (4, 2);
            \node at (3.5, 1.5) {$\mathbf{\textcolor{red}{2}}\ {\scriptstyle 1, 2}$};
            \node (star) at (3.5, 4.5) {$\star$};
            \node (r1) at (1.5, 4.5) {$\mathbf{\textcolor{red}{1}}$};
    
            \draw [->, thin, red] (star) -- (r1);
    
            \node at (4.5, 3.5) {$\mathbf{\cdots}$};
            \node at (2, 1 - 0.75) {$\mathbf{\vdots}$};
        \end{tikzpicture}
        \caption{Methods for relabeling an element of $H_2(\lambda)$.}
        \label{hrn2}
    \end{figure}
    
    To define the inverse map, we are given an element of $N_2(\lambda)$ with labels on two cells $(k, l), (i, j)$ where $k < i$ and $l < j$. 
    Consider the rectangle in the Young diagram containing the first $i$ rows and first $j$ columns. 
    Now look at whether the the cell $(k, l)$ is
    \begin{itemize}
        \item above the diagonal,
        \item below the diagonal,
        \item on the diagonal with label $1$, or
        \item on the diagonal with label $2$.
    \end{itemize}
    
    If it is above the diagonal, the star in the element of $H_2(\lambda)$ that we want to construct should be in cell $(i, k)$, the label in cell $(k, l)$ should be moved to $(i, l)$, and the label in cell $(i, j)$ should remain where it is. 
    One can check that this is indeed an element of $H_2(\lambda)$. 
    One can similarly define the inverse map in the other three cases.
\end{proof}

We now exhibit a bijection between the remaining elements of $H(\lambda)$ and those of $C(\lambda)$. 
We use $H_3(\lambda)$ to denote the set $H(\lambda) \setminus (H_1(\lambda) \cup H_2(\lambda))$. 
Note that the elements of $H_3(\lambda)$ are those elements of $H(\lambda)$ where
\begin{itemize}
    \item all the labeled cells are either in the same row or same column,
    \item the starred cell does not contain any other label, and
    \item all the labels lie in a single diagonal hook.
\end{itemize}

\begin{lemma}
    There is a bijection $\phi_3$ between $H_3(\lambda)$ and $C(\lambda)$.
\end{lemma}

\begin{proof}
    Starting with an element of $H_3(\lambda)$, if the labels are all in one row, then we swap the leftmost labeled cell with the rightmost one. 
    Similarly, if the labeled cells are all in one column, we swap the highest labeled cell with the lowest one.
    
    
    
    Since the star in an element of $H_3(\lambda)$ must be above (or to the left) of the labels $1, 2$ but the star in an element of $C(\lambda)$ must be below (or to the right) of the labels $1, 2$, the map described above is a bijection.
\end{proof}

\begin{figure}[H]
    \centering
    \ytableaushort
    {\none,\star, \none, 2, 1}
    * {3, 3, 2, 1, 1}
    \hspace{0.3cm}
    $\xrightarrow{\phi_3}$
    \hspace{0.2cm}
    \ytableaushort
    {\none, {\color{red} \textbf{1}}, \none, {\color{red} \textbf{2}}, {\color{red} \mathbf{\star}}}
    * {3, 3, 2, 1, 1}
    \hspace{1.5cm}
    \ytableaushort
    {\none,\none \star \none {\scriptstyle 1, 2}}
    * {5, 5, 2}
    \hspace{0.3cm}
    $\xrightarrow{\phi_3}$
    \hspace{0.2cm}
    \ytableaushort
    {\none, \none {\color{red}{\scriptstyle \textbf{1, 2}}} \none {\color{red} \mathbf{\star}}}
    * {5, 5, 2}
    \caption{Elements of $C(\lambda)$ associated to certain elements of $H_3(\lambda)$ under $\phi_3$.}
    \label{hc}
\end{figure}

\begin{proof}[Proof of \Cref{thm}]
    Combine the previous three lemmas with the fact that $H(\lambda) = H_1(\lambda) \sqcup H_2(\lambda) \sqcup H_3(\lambda)$ and $N(\lambda) = N_1(\lambda) \sqcup N_2(\lambda)$.
\end{proof}

\begin{remark}
    The inductive proof of \Cref{thm} in \cite{312771} is obtained by deleting an \emph{outside cell} of the partition $\lambda$. 
    This is a cell that has cells neither to its right nor below it. 
    Note that the deletion of such a cell gives us a partition, say $\mu$. 
    There is a natural way to see $H(\mu)$ as a subset of $H(\lambda)$ and similarly for $N(\mu)$ and $C(\mu)$. 
    This shows that our bijection also gives us a bijection between $H(\lambda) \setminus H(\mu)$ and $(N(\lambda) \setminus N(\mu)) \cup (C(\lambda) \setminus C(\mu))$. 
    Equating the sizes of these sets is precisely what is done in the proof given in \cite{312771}.
\end{remark}

\section{Rectangles in partitions}

\begin{definition}
    A \emph{thick rectangle} in (the Young diagram of) $\lambda$ is a rectangle made up of cells in $\lambda$ such that both the height and width of the rectangle are greater than $1$. 
    Other rectangles in $\lambda$ are called \emph{thin rectangles}.
\end{definition}

It is straightforward to see that the total number of rectangles in $\lambda$ is $\sum_{u = (i, j) \in \lambda} ij$ and that the number of thick rectangles is $\sum_{u = (i, j) \in \lambda} (i - 1)(j - 1)$. 
As a consequence of the map $\phi_2$ described in the previous section, we obtain another nice formula for the number of rectangles contained in the Young diagram of a partition.

The elements of $N_2(\lambda)$ described in the previous section correspond to (twice the number of) thick rectangles in the Young diagram of $\lambda$. 
The chosen cells in an element of $N_2(\lambda)$ form the north-west and south-east corners of the rectangle. 
Note that there are two ways to label the cells.

Also, one can check that the number of elements of $H_2(\lambda)$ with starred cell at $u$ is given by $\tilde{h}(u)^2$, where we define the \emph{partial hook length} of the cell $u = (i, j)$ as
\begin{equation*}
    \tilde{h}(u) =
    \begin{cases}
        0,\text{ if }i = j,\\
        \lambda_i - j,\text{ if }i > j,\text{ and}\\
        \lambda'_j - i,\text{ if }i < j.
    \end{cases}
\end{equation*}
The quantity $\lambda_i - j$ is usually denoted by $\operatorname{arm}(u)$ and $\lambda'_j - i$ by $\operatorname{leg}(u)$. 
Note that $h(u) = \operatorname{arm}(u) + \operatorname{leg}(u) + 1$.

It is also easy to see that the number of thin rectangles in $\lambda$ is $\sum_{u \in \lambda} h(u)$. 
Hence, from the above observations and \Cref{phi2} we get the following result.

\begin{result}
    For any $n \geq 1$ and $\lambda \vdash n$, the number of rectangles in the Young diagram of $\lambda$ is
    \begin{equation*}
        \sum_{u \in \lambda} h(u) + \frac{1}{2}\sum_{u \in \lambda} \tilde{h}(u)^2
    \end{equation*}
    where the first term counts thin rectangles and second counts thick rectangles.
\end{result}

\section{Other powers}

We first study what happens when we consider sums of higher powers of hook lengths and contents.

\begin{proposition}
    For any $n \geq 1$, $\lambda \vdash n$, and $k \geq 3$, we have
    \begin{equation*}
        \sum_{u \in \lambda} h(u)^k \leq n^k + \sum_{u \in \lambda} |c(u)|^k
    \end{equation*}
    with equality holding if and only if $\lambda$ is a hook (i.e., $\lambda = H((1, 1))$).
\end{proposition}


\begin{proof}
    We define $H^k(\lambda), N^k(\lambda)$ and $C^k(\lambda)$ analogously to \Cref{proofsec} (using labels in $[k] := \{1, 2, \ldots, k\}$ instead of just $1, 2$). 
    To prove the result, we use simple generalizations of the maps described in \Cref{proofsec}. 
    But we now show that these maps combine to an injection from $H^k(\lambda)$ into $N^k(\lambda) \cup C^k(\lambda)$ and that it is a bijection exactly when $\lambda$ is a hook.

    \noindent\textbf{Map $\phi_1^k$:}
    Start with an element of $H^k(\lambda)$ that either has
    \begin{itemize}
        \item at least one label from $[k]$ strictly to the right of the starred cell and at least one label strictly below the starred cell, or
        \item all labeled cells in the same row or column with the starred cell having at least one label from $[k]$.
    \end{itemize}
    To such an element of $H^k(\lambda)$, we associate the element of $N^k(\lambda)$ obtained by deleting the star.

    \noindent\textbf{Map $\phi_2^k$:}
    Start with an element of $H^k(\lambda)$ such that
    \begin{itemize}
        \item all the labeled cells are either in the same row or same column,
        \item the starred cell does not contain any other label, and
        \item all labels do \emph{not} lie in a single diagonal hook.
    \end{itemize}
    We now describe a method to associate an element of $N^k(\lambda)$ to such an element of $H^k(\lambda)$.

    First suppose that the element of $H^k(\lambda)$ does not have a cell labeled with all elements of $[k]$. 
    If the labels are in row $i$, in the cells $(i, j_0), (i, j_1), \ldots, (i, j_m)$ where $j_0 < j_1 < \cdots < j_m$. 
    Note that the starred cell is $(i, j_0)$. 
    Move the labels in cell $(i, j_l)$ to cell $(j_0, j_l)$ for all $l \in [m - 1]$ and delete the star. 
    The map for when all the labels are in the same column is defined analogously.

    Next suppose that the element of $H^k(\lambda)$ has a cell $(i, j)$ labeled with all elements of $[k]$. 
    If the starred cell is $(i, l)$, then we move the label $k$ from cell $(i, j)$ to cell $(l, l)$ and delete the star. 
    If the starred cell is $(l, j)$, then we move the label $1$ from cell $(i, j)$ to cell $(l, l)$ and delete the star.

    \noindent\textbf{Map $\phi_3^k$:}
    Just as for the $k = 2$ case, we now exhibit a bijection between the remaining elements of $H^k(\lambda)$ and the elements of $C^k(\lambda)$.
    
    Starting with such an element of $H^k(\lambda)$, if the labels are all in one row, then we swap the leftmost labeled cell with the rightmost one. 
    Similarly, if the labeled cells are all in one column, we swap the highest labeled cell with the lowest one.

    When $\lambda$ is a hook, the maps $\phi_1^k$ and $\phi_3^k$ cover all elements of $H^k(\lambda)$ and those of $N^k(\lambda) \cup C^k(\lambda)$ and hence gives us the required bijection. 
    If $\lambda$ is not a hook, it must contain the cell $(2, 2)$. 
    The element of $N^k(\lambda)$ that has the label $2$ in the cell $(1, 1)$ and the rest of the labels in cell $(2, 2)$ is in the image of neither $\phi_1^k$ nor $\phi_2^k$.
\end{proof}

We now consider the lower power $k = 1$.

\begin{proposition}\label{1power}
    For any $n \geq 1$, $\lambda \vdash n$, we have
    \begin{equation*}
        \sum_{u \in \lambda} h(u) \geq n + \sum_{u \in \lambda} |c(u)|
    \end{equation*}
    with equality holding if and only if $\lambda$ is a hook.
\end{proposition}

\begin{proof}
    We define $H^1(\lambda), N^1(\lambda)$ and $C^1(\lambda)$ just as in the proof of the previous proposition. 
    An injection from $N^1(\lambda) \cup C^1(\lambda)$ into $H^1(\lambda)$ is defined using similar ideas as before. 
    For an element of $N^1(\lambda)$, we obtain the corresponding element of $H^1(\lambda)$ by adding a star in the same cell that contains the label $1$. 
    For an element of $C^1(\lambda)$, we obtain the corresponding element of $H^1(\lambda)$ by swapping the star and the label $1$.

    It is easy to check that this is an injection. 
    The elements of $H^1(\lambda)$ that are in the image of this injection are those where the star and the label $1$ are in the same diagonal hook. 
    Such elements exhaust all of $H^1(\lambda)$ if and only if $\lambda$ is a hook.
\end{proof}

\section{Acknowledgements}

The author thanks Priyavrat Deshpande, Amrutha P, and Richard Stanley for their comments on this article. 
The author also thanks the referees for their helpful suggestions, which improved the quality of this paper. 
In particular, \Cref{1power} was suggested by a referee. 
This work was done while the author was at the Chennai Mathematical Institute, India and partially supported by a grant from the Infosys Foundation.


\end{document}